\documentclass[reqno]{amsart}
\usepackage[left=1in,right=1in,top=1in,bottom=1in]{geometry}
\usepackage{tikz}

\usetikzlibrary{shapes,decorations,calc,arrows}
\usepackage[foot]{amsaddr}
\usepackage{amsthm}
\usepackage{amscd}
\usepackage{amsfonts}
\usepackage{amsmath}
\usepackage{amssymb}
\usepackage{mathrsfs}
\usepackage{multirow}
\usepackage{verbatim}
\usepackage{url}
\usepackage[hidelinks]{hyperref}
\usepackage{graphicx}
\usepackage{cite}
\usepackage{fancyhdr}
\usepackage{yfonts}
\usepackage{setspace}
\usepackage{titlesec}
\usepackage{enumitem}
\usepackage{dsfont}

\titleformat{\section}[hang]
{\normalfont\Large\bfseries}
{\thesection.}{0.5em}{}

\titlespacing*{\section}{0pc}{2pc}{0.25pc}

\titleformat{\subsection}[runin]
{\normalfont\large\bfseries}
{\thesubsection}{0.5em}{}

\titlespacing{\subsection}{0pc}{1.5pc}{0.5pc}

\newcommand{\Aut}{\text{Aut}}

\newcommand{\N}{\mathbb{N}}
\newcommand{\Z}{\mathbb{Z}}

\newcommand{\R}{\mathbb{R}}
\newcommand{\C}{\mathbb{C}}
\renewcommand{\H}{\mathcal{H}}

\newcommand{\K}{\mathcal{K}}

\newcommand{\vphi}{\varphi}

\newcommand{\<}{\left\langle}
\renewcommand{\>}{\right\rangle}

\newcommand{\I}{\rm{I}}

\newcommand{\Ad}[1]{\text{Ad}\left(#1\right)}

\newcommand{\mr}[1]{\mathring{#1}}
\newcommand{\mt}[1]{\overset{\scriptscriptstyle\triangledown}{#1}}

\newcommand{\eig}{\text{eig}}

\newcommand{\Sd}{\operatorname{Sd}}
\renewcommand{\S}{\operatorname{S}}

\newcommand{\gp}{\mathop{\begin{tikzpicture}[baseline]
\node[circle, fill=cyan, draw=black, inner sep=0pt, minimum size=3pt] (mid) at (0, 2.5pt) {};
\foreach \x in {0,60,...,300} {
\node[circle, fill=cyan, draw, inner sep=0pt, minimum size=3pt] (\x) at ($(mid)!6pt!(mid.\x)$) {};
\draw (\x) -- (mid) ;
}
\end{tikzpicture}}}

\newtheorem{thm}{Theorem}[section]
\newtheorem{thmalpha}{Theorem}

\newtheorem{prop}[thm]{Proposition}
\newtheorem{lem}[thm]{Lemma}

\theoremstyle{definition}

\newtheorem{ex}[thm]{Example}

\newtheorem{rem}[thm]{Remark}

\title{\textbf{Irreducibility and weak spectral gap in free product von Neumann algebras}}
\author{Aldo Garcia Guinto$^1$}
\address{$^1$Department of Mathematics, University of Houston\hfill \url{aegarciaguin@central.uh.edu}}
\author{Fehmi Ekin Giritlioglu$^2$}
\address{$^2$Department of Mathematics, Michigan State University\hfill \url{giritlio@msu.edu}}
\author{Rahul Kumar R.$^3$}
\address{$^5$Department of Mathematics, Michigan State University \hfill \url{ramach32@msu.edu}}
\author{Yoonkyeong Lee$^4$}
\address{$^3$Department of Mathematics, Texas A\&M University\hfill \url{yoong@tamu.edu}}
\author{Brent Nelson$^5$}
\address{$^4$Department of Mathematics, Michigan State University \hfill\url{brent@math.msu.edu}}

\date{}

\begin{document}

\begin{abstract}
Consider a free product $(M,\varphi)= (M_1,\varphi_1)* (M_2,\varphi_2)$ of non-trivial von Neumann algebras. Using amalgamated free product techniques, we establish irreducibility and weak spectral gap results for free product subalgebras and the centralizer subalgebra of the free product state. In particular, it is shown that, under a mild dimension constraint, the inclusion of the diffuse summand of these subalgebras into the corresponding corner of $M^\omega$ is irreducible for any free ultrafilter $\omega\in \beta\N\setminus \N$. We also obtain analogous irreducibility results for graph products of von Neumann algebras.
\end{abstract}

\maketitle

\section*{Introduction}

For non-trivial von Neumann algebras $M_1$ and $M_2$ equipped with faithful normal states $\vphi_1$ and $\vphi_2$, respectively, their free product
    \[
        (M,\vphi):= (M_1, \vphi_1)* (M_2,\vphi_2)
    \]
is a fundamental construction that has been studied extensively (see, for example, \cite{Dyk93, Dyk94, Bar95, Dyk97, Hou07, IPP08, Ued11b, Ued13, Ued16, HN21}). The most general theorem is due to Ueda, and it states that when $\dim(M_1)+\dim(M_2)\geq 5$, then the diffuse summand of $M$ is a full factor of type $\rm{II}_1$ or $\rm{III}_\lambda$ with $\lambda\neq 0$ and the atomic summand of $M$ is finite dimensional and computable in terms of the finite dimensional summands (if any) of $M_1$ and $M_2$ (see \cite[Theorem 4.1]{Ued11a}). It is also known, thanks to \cite[Theorem 1.1]{Dyk93}, that this dimension condition is necessary for the diffuse summand of $M$ to be a factor. Moreover, when $M_1$ and $M_2$ are atomic, then the results of \cite{Dyk93, HN21} imply the diffuse summand is isomorphic to either an interpolated free group factor or a free Araki--Woods factor depending on whether or not $\vphi_1$ and $\vphi_2$ are both tracial. Yet even with so many advances, much remains to be understood about \emph{subalgebras} of free product von Neumann algebras, which is the focus of the present article.

Suppose for each $i=1,2$, that $N_i\leq M_i$ is a non-trivial subalgebra admitting a $\vphi_i$-invariant faithful normal conditional expectation $E_i\colon M_i\to N_i$. Then $N:=N_1\vee N_2 \leq M$ is also a free product von Neumann algebra (with respect to the states $\vphi_i|_{N_i}$, $i=1,2$), and so the aforementioned results imply its diffuse summand is a full factor if and only if $\dim(N_1)+\dim(N_2)\geq 5$. But more is true because of the ambient free product structure of $M$:

\begin{thmalpha}[{Theorem~\ref{thm:irreducibility_in_free_products}}]\label{introthm:A}
Consider the free product
    \[
        (M,\vphi):=(M_1,\vphi_1) *(M_2,\vphi_2).
    \]
Let $N:=N_1\vee N_2$ for non-trivial subalgebras $N_i\leq M_i$ admitting $\vphi_i$-preserving faithful normal conditional expectations, 
$i=1,2$, and let $z\in N$ be the central support of the diffuse summand of $N$. Then the following are equivalent:
    \begin{enumerate}[label=(\roman*)]
        \item $\dim(N_1)+\dim(N_2)\geq 5$;
        \item $Nz$ is a factor;
        \item $(Nz)'\cap (z M z)^\omega = \C z$ for every free ultrafilter $\omega\in \beta\N\setminus \N$.
    \end{enumerate}
\end{thmalpha}

\noindent In other words, not only is $Nz$ a full factor, but the inclusion $Nz \leq zM z$ is irreducible and has weak spectral gap. The above theorem can be partially deduced from the proof of \cite[Theorem 4.1]{Ued11a}, which is broken into cases according to whether or not $M_1$ and $M_2$ admit diffuse summands. However, it is not at all clear how to deduce Theorem~\ref{introthm:A} in the most challenging case: when $N_1$ and $N_2$ are both atomic. While the isomorphism class of $Nz$ is known in this case by \cite{Dyk93, HN21}, even this will not avail you as there is nothing to preclude an interpolated free group factor or a free Araki--Woods factor from embedding into $z Mz$ with non-trivial relative commutant. It may be possible to obtain Theorem~\ref{introthm:A} by sufficiently honing the methods in this final case, but we take a different approach. The key idea is to identify the free products with suitable amalgamated free products that admit sufficient diffuseness in one of the components. While our argument does ultimately rely on \cite[Theorem 4.1]{Ued11a} (at least insofar as asserting that $Nz$ is a factor), the amount of casework is sharply reduced through the use of amalgamated free products. We also remark that the technical results we obtain in the pursuit of Theorem~\ref{introthm:A} (see Proposition~\ref{prop:amalgmated_Ueda} and Theorem~\ref{thm:weak_spectral_amalgamated}) are similar to those obtained by Ueda in \cite{Ued13}, where some aspects of \cite{Ued11a} have been generalized to the setting of amalgamated free products. But our results are novel because we do not demand diffuseness relative to the amalgam, and this is crucial in the proof of Theorem~\ref{introthm:A}.

These same techniques can be applied to study the free product state $\vphi$ and its centralizer subalgebra $M^\vphi =\{ x\in M\colon \vphi(xy) = \vphi(yx)\ \forall y\in M\}$; in fact, the primary motivation for this article was to understand when the diffuse summand of $M^\vphi$ is a factor (or even full). It turns out that the notion of \emph{almost periodicity} for states, which was introduced by Connes in \cite{Con72} (see also \cite{Con74}), is essential to answering these questions. The \emph{almost periodic part} of $(M,\vphi)$ is the von Neumann subalgebra
    \[
        M^{(\vphi,\text{ap})}:=\{x\in M\colon \exists \lambda>0\ \varphi(xy) = \lambda \varphi(yx)\ \forall y\in M\}'',
    \]
which admits a $\vphi$-preserving conditional expectation (see \cite[Proposition 4.1]{GGLN25}), and one says $\vphi$ is \emph{almost periodic} if $M^{(\vphi, \text{ap})} = M$. Consider once more a free product $(M,\vphi) = (M_1,\vphi_1)*(M_2,\vphi_2)$ of non-trivial von Neumann algebras. Then $\vphi$ is almost periodic whenever both $\vphi_1$ and $\vphi_2$ are almost periodic, and in this case $M^\vphi$ and $M$ have the same central support $z$ for their diffuse summands (see Remark~\ref{rem:almost_periodic_diffuse_and_atomic} below). Moreover, another result of Ueda states that the same dimension condition $\dim(M_1)+\dim(M_2)\geq 5$ implies $(M^\vphi z)' \cap (Mz)^\omega = \C z$ (see \cite[Theorem 2.1]{Ued11b}). However, this can fail outside of the almost periodic case; for example, if $\vphi_1$ is ergodic in the sense that $M_1^{\vphi_1}=\C$ then $M^\vphi = M_2^{\vphi_2}$, and furthermore $M^\vphi=\C$ when both $\vphi_1$ and $\vphi_2$ are ergodic. Note that almost periodic states satisfy $(M^\vphi)'\cap M = (M^\vphi)' \cap M^\vphi$ (see \cite[Theorem 10]{Con72}), and so ergodicity and almost periodicity occur simultaneously only when $M=\C$. In fact, ergodicity is in some sense a negation of almost periodicity because it further implies the almost periodic part is trivial (see Example~\ref{ex:atomic_abelian_centralizer} below). Thus, any generalization of Ueda's result must necessarily avoid ergodicity of $\vphi_1$ and $\vphi_2$, and this is equivalent to asserting that the almost periodic parts $M_1^{(\vphi_1,\text{ap})}$ and $M_2^{(\vphi_2,\text{ap})}$ are non-trivial. Fortunately, little else is required:

\begin{thmalpha}[{Theorem~\ref{thm:full_centralizer}}]\label{introthm:B}
Consider the free product
    \[
        (M,\vphi):=(M_1,\vphi_1) *(M_2,\vphi_2),
    \]
where $\vphi_1$ and $\vphi_2$ are non-ergodic states. Let $z\in M^\vphi$ be the central support of the diffuse summand of $M^\vphi$. Then the following are equivalent:
    \begin{enumerate}[label=(\roman*)]
        \item $\dim(M_1^{(\vphi_1,\text{ap})})+\dim(M_2^{(\vphi_2,\text{ap})}) \geq 5$;
        \item $M^\vphi z$ is a factor;
        \item $(M^\vphi z)'\cap (z M z)^\omega= \C z$ for every free ultrafilter $\omega\in \beta\N\setminus \N$.
    \end{enumerate}
\end{thmalpha}

\noindent Since $M^{(\vphi,\text{ap})}$ is the free product of the almost periodic parts $M_1^{(\vphi_1,\text{ap})}$ and $M_2^{(\vphi_2,\text{ap})}$ (see Proposition~\ref{prop:almost_periodic_part_free_product} below), it follows that $z$ is also the central support of the diffuse summand of $M^{(\vphi,\text{ap})}$. Consequently, \cite[Theorem 4.1]{Ued11a} allows one to compute the finite dimensional corner $M^{(\vphi,\text{ap})}(1-z)$ and the restriction of $\vphi$, from which the subalgebra $M^\vphi(1-z)$ is readily determined. In particular, if the non-trivial, minimal, central projections in $M_1^{\vphi_1}$ and $M_2^{\vphi_2}$ all have mass at most $\frac12$, then $z=1$ (see Remark~\ref{rem:diffuse_free_product_centralizer}). In this case, $M^\vphi$ is diffuse, and if one further assumes $\dim(M_1^{(\vphi_1,\text{ap})})+\dim(M_2^{(\vphi_2,\text{ap})}) \geq 5$, then Theorem~\ref{introthm:B} implies it is a full type $\mathrm{II}_1$ factor whose inclusion into $M$ is irreducible and has weak spectral gap.

Both Theorems~\ref{introthm:A} and \ref{introthm:B} can be applied to give irreducibility results in the context of graph product von Neumann algebras (see Theorems~\ref{thm:irreducible_graph_products} and \ref{thm:extremal_graph_products}). We also establish instances in which the factoriality of the centralizer of a graph product of almost periodic states can be fully characterized (see Theorems~\ref{thm:extremal_atomic_graph_products} and \ref{thm:extremal_almost_periodic_graph_products}). In many of these situations, it is therefore possible to deduce the resulting (sub)type of the graph product von Neumann algebra (see Proposition~\ref{prop:join-irreducible_extremal} and Remarks~\ref{rem:graph_product_type} and \ref{rem:almost_periodic_graph_product_type}).

\section*{Acknowledgments}
We thank Michael Hartglass, Ben Hayes, David Jekel, and Srivatsav Kunnawalkam Elayavalli for helpful discussions related to this work. No large language models were utilized at any point in the development of this article. All authors were supported by NSF grant DMS-2247047.

\section{Preliminaries}

Given a von Neumann algebra $M$, we will make use of lattice notation for the collection of (unital) von Neumann subalgebras: we write $N\leq M$ to denote that $N$ is a von Neumann subalgebra of $M$ and we write $N_1 \vee N_2$ for the von Neumann algebra generated by $N_1, N_2\leq M$. A von Neumann algebra $M$ is called \textit{diffuse} if it admits no minimal projections, and $M$ is called \textit{atomic} if it is a direct sum of type $\mathrm{I}$ factors. There always exists a unique central projection $z\in M$ so that $Mz$ is diffuse and $M(1-z)$ is atomic. Indeed, $1-z$ is merely the sum of any atoms in the center of the type $\mathrm{I}$ summand of $M$.

We will assume some familiarity with amalgamated free products and graph products of von Neumann algebras. For the background on amalgamated free products, we direct the reader to \cite[Section 2]{Ueda99}. For background on graph products, we direct the reader to \cite[Section 1.1]{CdSHJKEN25} and \cite[Section 3]{CF17}.

For a faithful normal state $\varphi$ on $M$, we refer to the pair $(M,\vphi)$ as a \emph{statial} von Neumann algebra. We write $L^2(M,\vphi)$ for the GNS Hilbert space associated to this pair and we will implicitly identify $M$ with its representation on this Hilbert space. The map $x\mapsto x^*$ is a densely defined and closable operator on $L^2(M,\vphi)$, whose closure we denote by $S_\vphi$. Letting $S_\vphi = J_\vphi \Delta_{\vphi}^{1/2}$ be the polar decomposition, then $J_\vphi$ is a conjugate linear, involutive, isometry called the \emph{modular conjugation} of $\vphi$, and $\Delta_\vphi$ is a non-singular, positive, self-adjoint operator called the \emph{modular operator} of $\vphi$. These operators satisfy
    \[
        J_\vphi M J_\vphi = M' \cap B(L^2(M,\vphi)) \qquad \text{ and } \qquad \Delta_\vphi^{it} M \Delta_\vphi^{-it} = M \qquad t\in \R.
    \]
Consequently, $\sigma_t^\vphi:= \Ad{\Delta_\vphi^{it}}\in \Aut(M)$ for all $t\in \R$, and the strongly continuous action $\sigma^\vphi\colon \R\curvearrowright M$ is called the \emph{modular automorphism group} of $\vphi$. The fixed point subalgebra
    \[
        M^\vphi:= \{x\in M\colon \sigma_t^{\vphi}(x) = x\ \forall t\in \R\}
    \]
is called the \emph{centralizer} of $\vphi$, and we say $\vphi$ is \emph{ergodic} if $M^\vphi = \C$ (in reference to the ergodicity of $\sigma^\vphi$). We say a non-zero element $x\in M$ is an eigenoperator of $\sigma^\vphi$ if there exists $\lambda >0$ so that $\sigma_t^\vphi(x)=\lambda^{it} x$ for all $t\in \R$. The set of eigenoperators with eigenvalue $\lambda$ is denoted $M^{(\vphi,\lambda)}$ (so that $M^{(\vphi,1)}= M^\vphi\setminus \{0\}$),
and thanks to the KMS condition satisfied by $\sigma^\vphi$ (see \cite[Section VIII.1]{Tak03}) one can also characterize this space as follows:
    \[
        M^{(\vphi,\lambda)}=\{x\in M\setminus\{0\}\colon \vphi(xy) = \lambda \vphi(yx)\ \forall y \in M\}.
    \]
We denote $\Sd(\vphi):=\{\lambda>0\colon M^{(\vphi,\lambda)}\neq\emptyset\}$, which is equivalently the point spectrum of $\Delta_\vphi$. We also define
    \[
        M^{(\vphi,\eig)}:= \text{span}\left(\bigcup_{\lambda \in \Sd(\vphi)} M^{(\vphi,\lambda)}\right),
    \]
which one can see is a unital $*$-subalgebra by verifying that $M^{(\vphi,\lambda_1)}M^{(\vphi,\lambda_2)} = M^{(\vphi, \lambda_1\lambda_2)}$ and $(M^{(\vphi,\lambda)})^* = M^{(\vphi, 1/\lambda)}$. One says $\vphi$ is \emph{almost periodic} if $(M^{(\vphi,\eig)})'' =M$. This is equivalent to $\Delta_\vphi$ being diagonalizable (see \cite[Lemma 1.4]{GGLN25}), and is further equivalent to Connes' original definition that for each pair $x,y\in M$ the function
    \[
        \R\ni t\mapsto \vphi(\sigma_t^{\vphi}(x)y)
    \]
is almost periodic in the sense that it is a uniform limit of linear combinations of periodic functions (see \cite[Lemma 7]{Con72}). In general, one can consider the \emph{almost periodic part} of $(M,\vphi)$, which is given by
    \[
        M^{(\vphi,\text{ap})}:= (M^{(\vphi,\eig)})''.
    \]
This subalgebra is always the range of a $\vphi$-preserving faithful normal conditional expectation,  and the restriction $\vphi|_{M^{(\vphi,\text{ap})}}$ is always an almost periodic state (see \cite[Proposition 4.1]{GGLN25}).

\begin{rem}\label{rem:almost_periodic_diffuse_and_atomic}
Let $(M,\vphi)$ be a von Neumann algebra equipped with an almost periodic state, and let $z\in M$ be the central projection such that $Mz$ is diffuse and $M(1-z)$ is atomic. Then $M^{\vphi}z$ is diffuse by \cite[Lemma 2.1]{Ued11b}, and we further claim that $M^{\vphi}(1-z)$ is atomic. Indeed, $M(1-z)$ is a countable direct sum of type $\I$ factors, and consequently the restriction of $\vphi$ to this corner is the sum of functionals given by taking traces against positive trace-class operators. Thus $M^{\vphi}(1-z)$ corresponds to the direct sum of commutants of these positive trace-class operators. Since positive trace-class operators are, in particular, compact, they have finite dimensional eigenspaces and therefore their commutant is a direct sum of matrix algebras. Hence $M^{\vphi}(1-z)$ is atomic.$\hfill\blacksquare$
\end{rem}

We now collect a number of results that are quite likely to be well-known to experts but will be convenient to reference in later sections of the article.

\begin{lem}\label{lem:ergodic_states}
If $\vphi$ is an ergodic state on $M$, then $M$ is either trivial or a type $\mathrm{III}$ factor
\end{lem}
\begin{proof}
Since $M'\cap M \subset M^\vphi$, it follows that $M$ is a factor. If $M$ is not type $\mathrm{III}$, then it is semifinite and it follows from \cite[Remark 1.1]{GGLN25} that $\vphi$ is tracial. But then $M=M^\vphi =\C$.
\end{proof}

\begin{lem}\label{lem:irred_corners}
Let $N\leq M$ and let $p\in N$ be a projection with central support $z$ in $N$. Then $ p N p \leq p M p$ is irreducible if and only if $ N z \leq z M z$ is irreducible.
\end{lem}
\begin{proof}
We fix a family $\{v_i\in N z\colon i\in I\}$ of partial isometries satisfying $v_i^* v_i \leq p$ for all $i\in I$ and $\sum_i v_i v_i^* =z$. We can (and do) arrange that $p=v_i$ for some $i\in I$.

Suppose $p N p \leq p M p$ is irreducible. For $x\in (Nz)'\cap z M z$, we have $pxp = xp \in (pNp)'\cap p M p =\C p$, say with $pxp = \alpha p$. Then
    \[
        x = \sum_{i\in I} x v_i v_i^* = \sum_{i\in I} v_i x v_i^* = \sum_{i\in I} v_i p x p v_i^* = \sum_{i\in I} v_i \alpha v_i^* = \alpha z,
    \]
so that $Nz \leq z M z$ is irreducible.

Conversely, if $Nz \leq z M z$ is irreducible, then it is straightforward to check that for any $x\in (pNp)'\cap pMp$ the element
    \[
        \sum_{i\in I} v_i x v_i^*
    \]
lies in $(Nz)'\cap (zMz) = \C z$. Compressing by $p$ then yields that $x\in \C p$, and so $p N p \leq pMp$ is irreducible.
\end{proof}

\begin{lem}\label{lem:amalgamated_osmosis}
Consider an amalgamated free product
    \[
        (M,E) = (M_1, E_1) \underset{A}{*} (M_2,E_2).
    \]
Let $A\leq B \leq M_1$ be an intermediate algebra with an $E_1$-preserving faithful normal conditional expectation $F_1\colon M_1\to B$. Then there exists a unique $E$-preserving faithful normal conditional expectation $F\colon M\to B$ and one has 
    \[
        (M,F)\cong (M_1, F_1) \underset{B}{*} ( M_2 \vee B, F|_{M_2\vee B}).
    \]
\end{lem}
\begin{proof}
Let $E^1\colon M\to M_1$ be the unique faithful normal conditional expectation satisfying $E_1\circ E^1 = E$. Then $F:=F_1\circ E^1$ is a faithful normal conditional expectation from $M$ onto $B$ satisfying
    \[
        E\circ F = E \circ F_1 \circ E^1 = E_1 \circ F_1 \circ E^1 = E_1\circ E^1 = E.
    \]
If $F'\colon M\to B$ also satisfies $E\circ F'= E$, then for any $x\in M$ setting $b:= F[x] - F'[x]$ one has
    \[
        E[ b^* b] = E[ F[b^*x] - F'[b^*x]] = E[b^*x] - E[b^* x]=0,
    \]
so that $b=0$ by the faithfulness of $E$. Hence $F=F'$ is unique.

Now, $M$ is generated by $M_1$ and $M_2\vee B$, and so in order to identify it with the amalgamated free product over $B$ it suffices to show that $M_1$ and $M_2\vee B$ are free with amalgamation with respect to $F$. Using $E\circ F=E$, we see that $(M_2\vee B)\cap \ker{F} \subset (M_2 \vee B)\cap \ker{E}$, and so the former set is spanned by alternating products of elements centered with respect to $E$. An alternating product of such words and elements from $M_1\cap \ker{F_1}$ can then be reorganized as an alternating product of elements from $M_1\cap \ker{F_1}\subset M_1\cap \ker{E_1}$ and $M_2\cap \ker{E_2}$, which lies in $\ker{E}$. Denote such an alternating product by $x$, and note $b^*x$ still lies in $\ker{E}$ for any $b\in B\cap \ker{E}$ since $x=ay$ with either $a\in M_1\cap \ker{F_1}$ so that $E_1[b^* a] = E_1[b^* F_1[a]]=0$, or $a\in M_2\cap \ker{E_2}$ and $E[b^* x]=0$ by freeness over $A$. Since $E[F[x]]= E[x]=0$, we can take $b:=F[x]$ and we see that
    \[
        E[ F[x]^* F[x]]= E[ F[b^* x]] = E[b^*x]=0.
    \]
Since $E$ is faithful, it follows that $F[x]=0$.
\end{proof}

\section{Free product subalgebras}

We begin with a mild generalization of \cite[Proposition 3.5]{Ued11a} to the operator-valued setting. The overall structure of the argument is nearly identical, but there are enough small changes to warrant including a fully detailed proof.

\begin{prop}\label{prop:amalgmated_Ueda}
Consider the amalgamated free product
    \[
        (M,E)=(M_1, E_1)\underset{B}{*}(M_2,E_2).
    \]
Let $F\colon M\to B$ be another $E^1$-invariant faithful normal conditional expectation, where $E^1:M\rightarrow M_1$ is the unique conditional expectation satisfying $E_1\circ E^1=E$. Let $\psi\in M_*$ be an $F$-invariant faithful normal state. Suppose there exists a projection $p\in B^\psi$ and sequences $(u_n)_{n\in\N}, (v_n)_{n\in \N} \subset M_1^\psi$ satisfying $u_n^*u_n=u_nu_n^*=v_n^*v_n=v_nv_n^*=p$ and $E[u_n]=E[u_n^*u_m]=F[v_n]=F[v_n^*v_m]=0$ for all $n\neq m$. Then for any free ultrafilter $\omega\in \beta\N\setminus \N$ and any $x\in \{u_1,v_1,u_2,v_2,\ldots \}'\cap (pMp)^\omega$ one has
    \begin{align}\label{Eqn:comm_bound_on_distance_to_subalg}
        \| y( x- E^{1,\omega}[x])\|_{\psi^\omega} \leq \| [x,y] \|_{\psi^\omega} \qquad y\in M_2\cap \ker{E},
    \end{align}
where $E^{1,\omega}\colon M^\omega \to M_1^\omega$ is the canonical lifting of $E^1$. In particular, if there exists $y\in pM_2 p \cap \ker{E}$ that is invertible in $p M_2 p$, then $\{y,u_1,v_1,u_2,v_2,\ldots\}'\cap (pMp)^\omega \subset (pBp)^\omega$.
\end{prop}
\begin{proof}
Denote $ \mr M_i:=M_i\cap \ker E$ for $i=1,2$ and $\mt{M_1}:=M_1 \cap \ker F$. Let $\H_{1,1}, \H_{1,2}, \H_{2,1},\H_{2,2}\subset L^2(M,\psi)$ be the closed subspaces generated by  
     \[
        \underbrace{\mr{M_1} \mr{M_2}\cdots \mt{M_1}}_{\text{length}\geq 3},\qquad \underbrace{\mr{M_1} \mr{M_2}\cdots \mr{M_2}}_{\text{length}\geq 2}, \qquad \underbrace{\mr{M_2} \mr{M_1}\cdots \mt{M_1}}_{\text{length}\geq 2}, \qquad \underbrace{\mr{M_2} \mr{M_1}\cdots \mr{M_2}}_{\text{length}\geq 1},
    \]
respectively. Note that the product of the adjoint of any of the above (non-closed) subspaces against any of the other subspaces lies in the kernel of $F$; for example, the inclusion
    \[
        (\mr{M_1} \mr{M_2}\cdots \mr{M_2})^*\mr{M_1} \mr{M_2}\cdots \mt{M_1} = \mr{M_2} \cdots \mr{M_2} (\mr{M_1}\mr{M_1}) \mr{M_2} \cdots \mt{M_1} \subset \mt{M_1} + \ker{E^1} \subset \ker{F}
    \]
can be seen by successively centering the middle terms with respect to $E$ and noting that $\mr{M_2}\cdots \mr{M_2}\mt{M_1}\subset \ker{E^1}$. Consequently, one has
    \begin{align*}
        L^2(M,\psi)=L^2(M_1,\psi)\oplus\H_{1,1}\oplus \H_{1,2}\oplus \H_{2,1}\oplus \H_{2,2}.
    \end{align*}
 Set $\K_{1,2}^{(n)}:=v_nJ_\psi v_n J_\psi\H_{1,2},\quad  \K_{2,1}^{(n)}:=u_nJ_\psi u_n J_\psi\H_{2,1}$, and $\K_{2,2}^{(n)}:=u_nJ_\psi u_n J_\psi\H_{2,2}$. Our assumptions on $(u_n)_{n\in \N}$ and $(v_n)_{n\in \N}$ imply that $\K_{i,j}^{(n)}$ forms a sequence of orthonormal subspaces for every $(i,j)\in \{(1,2),(2,1),(2,2)\}$. Let $P_{i,j}, Q_{i,j}^{(n)}\subset B(L^2(M,\psi))$ be the orthogonal projection onto $\H_{i,j}$ and $\K_{i,j}^{(n)}$ for $(i,j)\neq (1,1)$ and $n\in \N$ respectively. Since each $\H_{i,j}$ is a $B$-bimodule, one has $[pJ_\psi pJ_\psi, P_{i,j}]=0$ and it follows that $(v_nJ_\psi v_n J_\psi)\circ P_{1,2}\circ (v_n^*J_\psi v_n^*J_\psi)=Q_{1,2}^{(n)}$ and $(u_nJ_\psi u_n J_\psi)\circ P_{i,j}\circ (u_n^*J_\psi u_n^*J_\psi)=Q_{i,j}^{(n)}$ for $(i,j)\in \{(2,1), (2,2)\}$.
 
Fix arbitrary $x=(x_k)_\omega\in \{u_1,v_1,u_2,v_2\ldots, u_n,v_n,\ldots \}'\cap (pMp)^\omega$. Then one has the following:
\begin{align*}
    \lim_{k\rightarrow \omega}\|P_{1,2}(x_k)\|_{\psi}&=\lim_{k\rightarrow\omega}\|pJ_\psi pJ_\psi P_{1,2}(px_k)\|_{\psi}\\
&=\lim_{k\rightarrow\omega}\|v_nJ_\psi v_nJ_\psi P_{1,2}(v_n^*v_nx_k)\|_{\psi}\\
&\leq \lim_{k\rightarrow\omega}\|Q_{1,2}^{(n)}(x_k)\|_\psi+\|v_nJ_\psi v_nJ_\psi P_{1,2}(v_n^*[v_n,x_k])\|_\psi\\
&=\lim_{k\rightarrow\omega}\|Q_{1,2}^{(n)}(x_k)\|_\psi,
\end{align*}
where the last equality follows since $\underset{k\rightarrow \omega}{\lim}\|v_nJ_\psi v_nJ_\psi P_{1,2}(v_n^*[v_n,x_k])\|_\psi \leq \underset{k\rightarrow \omega}{\lim}\|[v_n,x_k]\|_\psi=\| [v_n,x]\|_{\psi^\omega}=0$. Similarly, by replacing the role of $v_n$ with $u_n$, it follows that $ \underset{k\rightarrow \omega}{\lim}\|P_{2,1}(x_k)\|_{\psi}\leq \underset{k\rightarrow \omega}{\lim}\|Q_{2,1}^{(n)}(x_k)\|_\psi$ and $ \underset{k\rightarrow \omega}{\lim}\|P_{2,2}(x_k)\|_{\psi}\leq \underset{k\rightarrow \omega}{\lim}\|Q_{2,2}^{(n)}(x_k)\|_\psi$. Hence for $(i,j)\neq (1,1)$, we have
\begin{align*}
  \|x\|_{\psi^\omega}^2=\lim_{k\rightarrow \omega}\|x_k\|_{\psi}^2 \geq \lim_{k\rightarrow \omega}\sum_{n\in \N}\|Q_{i,j}^{(n)}(x_k)\|_{\psi}^2\geq \sum_{n\in \N}\lim_{k\rightarrow \omega}\|Q_{i,j}^{(n)}(x_k)\|_{\psi}^2 \geq \sum_{n\in \N}\lim_{k\rightarrow \omega}\|P_{i,j}(x_k)\|_{\psi}^2,
\end{align*}
which in turn implies that 
\begin{align}\label{UltraPdt Projection}
\lim_{k\rightarrow \omega}\|P_{i,j}(x_k)\|_{\psi}=0 \text{ for } (i,j)\neq (1,1).
\end{align}
Now identify $L^2(M^\omega, \psi^\omega)$ as a closed subspace of $L^2(M,\psi)^\omega$ (see \cite[Section 2.2]{Ued11a}). Then the above equation implies  
\begin{align*}
    x-E^{1,\omega}[x]=(P_{1,1}(x_k))_{k\in \omega},
\end{align*}
where $E^{1,\omega}:M^\omega\rightarrow M_1^\omega$ is the canonical lifting of $E^1$ (or the  unique $\psi^\omega$-preserving conditional expectation).

Next we show that $y(x-E^{1,\omega}[x])$ is orthogonal to $[y,E^{1,\omega}[x]]$ and $(x-E^{1,\omega}[x])y$ for $y \in\mr{ M_2}$. Fix $y \in \mr{M_2}$. We have $y(x-E^{1,\omega}[x])=(yP_{1,1}(x_k))_{k\in \omega}$ and $yP_{1,1}(x_k)$ lies in the closed linear span of $\underbrace{\mr{M_2}\mr{M_1} \mr{M_2}\cdots \mt{M_1}}_{\text{length}\geq 4}$ for $k\in \omega$. Notice that $[y,E^{1,\omega}[x]]=yE^{1,\omega}[x]-E^{1,\omega}[x]y=(yE^1(x_k)-E^1(x_k)y)_{k\in\omega}$, and for $k\in \omega$ we have 
\begin{align*}
   yE^1(x_k)-E^1(x_k)y &\in(\mr M_2M_1-M_1\mr M_2)\subset \mr M_2 \oplus \mr{M_2}\mt{M_1}\oplus \mr{M_1}\mr{M_2},
\end{align*}
by centering the first instance of $E^1(x_k)$ term with respect to $F$ and the second with respect to $E$. Hence, we have $y(x-E^{1,\omega}[x])$ is orthogonal to $[y,E^{1,\omega}[x]]$. We proceed to show the orthogonality of $y(x-E^{1,\omega}[x])$ and $(x-E^{1,\omega}[x])y$ using the Connes cocycle derivative between $\psi$ and the state $\vphi:=\psi|_B\circ E$. Set 
    \[
        y_n:=\sqrt{\frac{n}{\pi}}\int_\R e^{-n t^2}\sigma_t^\psi(y) dt =\sqrt{\frac{n}{\pi}}\int_\R e^{- nt^2}(D_\psi : D_\varphi)_t \sigma_t^\varphi(y)(D_\psi : D_\varphi)_t^*dt.
    \]
Note that unlike $\sigma^\psi$, the modular automorphism group $\sigma^\vphi$ leaves $M_2$ globally invariant since $\vphi\circ E^2 = \vphi$ for $E^2\colon M\to M_2$ the unique conditional expectation satisfying $E_2\circ E^2 = E$. Moreover, since $\vphi|_{M_2} = \psi|_B\circ E_2$, it follows that $\sigma_t^{\vphi}(y)=\sigma_t^{\vphi|_{M_2}}(y)\in \mr{M_2}$. Additionally, using $\vphi = \psi|_B\circ E_1\circ E^1 = \vphi|_{M_1}\circ E^1$ and $\psi = \psi|_{M_1}\circ E^1$, it follows from \cite[Corollary IX.4.22.(ii)]{Tak03} that
    \[
        (D_\psi: D_\varphi)_t = (D_{\psi|_{M_1}} : D_{\varphi|_{M_1}})_t\in M_1
    \]
for all $t\in \R$. Thus $y_n$ belongs to the $\sigma$-weak operator topology closure of the linear span of $M_1\mr{M_2}M_1$. Additionally, each $y_n$ is analytic with respect to $\sigma^\psi$ and  $y_n \to y$  in the $\sigma$-weak operator topology as $n \rightarrow \infty$. For each $n\in \N$ we have 
    \begin{align*}
        \|(x-E^{1,\omega}[x])y_n- &(J_\psi \sigma_{-i/2}^\psi(y_n^*)J_\psi P_{(1,1)}(x_k))_\omega\|_{\psi^\omega}\\
            &=\lim_{k\rightarrow\omega}\|(x_k-E^1[x_k])y_n- J_\psi \sigma_{-i/2}^\psi(y_n^*)J_\psi P_{(1,1)}(x_k)\|_{\psi}\\
            &=\lim_{k\rightarrow\omega}\|(J_\psi\sigma_{-i/2}^\psi(y_n^*)J_\psi)(x_k-E^1[x_k]-P_{(1,1)}(x_k))\|_{\psi}\\
            &\leq\|J_\psi \sigma_{-i/2}^\psi(y_n^*)J_\psi\|\lim_{k\rightarrow\omega}\|P_{(1,2)}(x_k)+P_{(2,1)}(x_k)+P_{(2,2)}(x_k)\|_{\psi}.
\end{align*}
Hence it follows from \eqref{UltraPdt Projection} that $(x-E^{1,\omega}[x])y_n= (J_\psi\sigma_{-i/2}^\psi(y_n^*)J_\psi P_{(1,1)}(x_k))_\omega$. Further we note that\\
\begin{align*}
  J_\psi \sigma_{-i/2}^\psi(y_n^*)J_\psi P_{(1,1)}(x_k)&\in J_\psi \sigma_{-i/2}^\psi(y_n^*)J_\psi \overline{\text{span}}\{\mr{M_1} \mr{M_2}\cdots \mt{M_1}\}\\
  &\subset \overline{\text{span}}\{\mr{M_1} \mr{M_2}\cdots \mt{M_1}y_n\}\\
  &\subset \overline{\text{span}}\{\mr{M_1} \mr{M_2}\cdots \mt{M_1} (M_1\mr{M_2}M_1)\} \\
  &\subset L^2(M_1,\psi) \oplus \H_{1,1} \oplus \H_{1,2},
\end{align*}
where the last containment follows by centering the middle terms with respect to $E$ and centering the last $M_1$ term with respect to $F$. As noted above, $y P_{1,1}(x_k) \in \H_{2,1}$, and so we can conclude that $y(x-E^{1,\omega}[x]) = (y P_{1,1}(x_k))_{k\in \omega}$ and $(x-E^{1,\omega}[x])y_n$ are orthogonal for all $n\in \N$. But then the orthogonality of $y(x-E^{1,\omega}[x])$ and $(x-E^{1,\omega}[x])y$ follows from the normality of $\psi^\omega$:
\begin{align*}
    \< (x-E^{1,\omega}[x])y, y(x-E^{1,\omega}[x])\>_{\psi^\omega} &= \psi^\omega( (x - E^{1,\omega}[x])^* y^* (x-E^{1,\omega}[x])y)\\
        &= \lim_{n\to\infty} \psi^\omega( (x - E^{1,\omega}[x])^* y^* (x-E^{1,\omega}[x])y_n)\\
        &= \lim_{n\to\infty} \< (x-E^{1,\omega}[x])y_n, y(x-E^{1,\omega}[x])\>_{\psi^\omega}=0.
\end{align*}
Consequently, we have
    \begin{align*}
        \| y(x - E^{1,\omega}[x]) \|_{\psi^\omega}^2 &= \left| \< (x- E^{1,\omega}[x]) y - [x,y] - \left[y, E^{1,\omega}[x]\right], y(x- E^{1,\omega}[x]) \>_{\psi^\omega} \right|\\
            &= \left| \< [x,y], y(x- E^{1,\omega}[x])\>_{\psi^\omega} \right| \leq \| [x,y] \|_{\psi^\omega} \| y(x- E^{1,\omega}[x]) \|_{\psi^\omega},
    \end{align*}
which yields \eqref{Eqn:comm_bound_on_distance_to_subalg}.

Finally, suppose there exists $y\in p M_2 p\cap \ker{E}$ and $z\in p M_2p$ satisfying $zy=p$. If $x$ also commutes with this $y$, then \eqref{Eqn:comm_bound_on_distance_to_subalg} implies
    \[
        \|x - E^{1,\omega}[x]\|_{\psi^\omega} = \| p(x- E^{1,\omega}[x])\|_{\psi^\omega} \leq \|z\| \| y(x- E^{1,\omega}[x])\|_{\psi^\omega} \leq \|z\| \|[x,y]\|_{\psi^\omega} =0.
    \]
Thus $x = E^{1,\omega}[x]\in M_1^\omega$, so we can find a new representative sequence $x=(a_k)_{k\in \omega}$ with each $a_k\in M_1$ (namely, $a_k= E_1[x_k]$). Let $F^\omega\colon M^\omega \to B^\omega$ be the canonical lifting of $F$ (or the unique $\psi^\omega$-preserving conditional expectation), and observe that $y(x - F^\omega[x])=( y(a_k - F[a_k]))_{k\in \omega}$ has a representative sequence from $\mr{M_2}\mt{M_1}$. On the other hand, $[y, F^\omega[x]] = ( [y, F[a_k]])_{k\in \omega}$ has a representative sequence from $\mr{M_2}$ and $(x- F^\omega[x])y = ( (a_k - F[a_k]) y)_{k\in \omega}$ has a representative sequence from $\mt{M_1}\mr{M_2}\subset \mr{M_2}\oplus \mr{M_1}\mr{M_2}$. It follows that $y(x - F^\omega[x])$ is orthogonal to both $[y, F^\omega[x]]$ and $(x - F^\omega[x])y$, and using $[x,y]=0$ we have
    \begin{align*}
       \|x - F^\omega[x]\|_{\psi^\omega}^2 \leq \|z\|^2 \|y(x- F^\omega[x])\|_{\psi^\omega}^2 = \|z\|^2 \left| \<  (x- F^\omega[x])y  - [y, F^\omega[x]],  y(x- F^\omega[x])\>_{\psi^\omega}\right|=0.
    \end{align*}
Thus $x=F^\omega[x]\in (pBp)^\omega$ as claimed.
\end{proof}


The next result will be our main technical tool for proving Theorems~\ref{introthm:A} and \ref{introthm:B}. We note that the case when $B=\C$ and $p=1$ is well-known (see, for example, \cite[Theorem 3.7]{Ued11a}).

\begin{thm}\label{thm:weak_spectral_amalgamated}
Consider the amalgamated free product
    \[
        (M,E) = (M_1, E_1) \underset{B}{*} (M_2,E_2),
    \]
For each $i=1,2$, let $B\leq N_i\leq M_i$ be an intermediate algebra with an $E_i$-preserving faithful normal conditional expectation $F_i\colon M_i\to N_i$. Suppose there exists a minimal and central projection $p\in B$ so that $p N_1 p$ is diffuse and $p N_2 p\neq \C p$. Then $[p(N_1 \vee N_2) p]' \cap (p M p)^\omega = \C p$ for any free ultrafilter $\omega\in \beta \N\setminus \N$.
\end{thm}
\begin{proof}
Fix any $E$-invariant faithful normal state $\vphi\in M_*$. Our aim is to apply Proposition~\ref{prop:amalgmated_Ueda}, but we will first construct a faithful normal state $\psi$ with suitable properties and then use modular theory techniques to produce a $\psi$-preserving conditional expectation $G\colon M\to B$. Toward this end, we fix a faithful normal state $\psi_0$ on $p N_1 p$ with diffuse centralizer $(p N_1 p)^{\psi_0}$, which exists by the argument in the proof of \cite[Theorem 3.4]{Ued11a}. Extend $\psi_0$ to a faithful normal state on $p Mp$ via $\psi_1:=\psi_0\circ F_1\circ E^1$, where, as usual, $E^1\colon M\to M_1$ denotes the unique conditional expectation satisfying $E_1\circ E^1 = E$. Then define a normal state $\psi$ on $M$ by
    \[
        \psi(x):= \varphi(p)\psi_1(p x p) + \varphi( (1-p)x (1-p)) \qquad x\in M.
    \]
To see that $\psi$ is faithful, observe that for positive $x\in M$ one has $\psi(x)=0$ if and only if $\psi_1(pxp)=0$ and $\varphi((1-p)x(1-p))=0$. The faithfulness of $\psi_0$ and $\varphi$ imply $pxp=0=(1-p)x(1-p)$, and so
    \[
        \varphi(x) = \varphi( xp + x(1-p)) = \varphi(pxp) + \varphi((1-p)x(1-p))=0,
    \]
where we have used that $p\in B'\cap B\subset B^\vphi \subset M^\varphi$. Hence the faithfulness of $\varphi$ implies $x=0$ and therefore $\psi$ is faithful. Next, we claim that $B$ is invariant under $\sigma^\psi$. Indeed, it follows from the definition of $\psi$ that $p\in M^\psi$ and that $\sigma^\psi|_{B(1-p)} = \sigma^{\vphi}|_{B (1-p)}$, thus $B=\C p \oplus B(1-p)$ is $\sigma^\psi$-invariant because both summands are invariant. This claim along with \cite[Theorem IX.4.2]{Tak03} yields a $\psi$-preserving faithful normal conditional expectation $G\colon M\to B$. Note that $G$ is $E^1$-invariant since $\psi$ is $E^1$-invariant (by virtue of $\psi_1$ and $\varphi$ each being $E^1$-invariant).

Now, recalling that $p N_1^\psi p=  (p N_1 p)^{\psi_0}$ is diffuse, we can argue as in \cite[Theorem 3.7]{Ued11a} to find unitaries $u,v\in p N_1^\psi p$ satisfying $\varphi(u^n)=0 = \psi(v^n)$ for all $n\in \Z\setminus\{0\}$. Using the $E$-invariance and $G$-invariance of $\vphi$ and $\psi$, respectively, it follows that
    \[
        E[x] = \frac{\varphi(x)}{\varphi(p)} p  \qquad \text{ and } \qquad G[x] = \frac{\psi(x)}{\psi(p)} p \qquad\qquad x\in p M p.
    \]
Thus the sequences $(u^n)_{n\in \N}, (v^n)_{n\in \N} \subset p N_1 ^\psi p \subset M_1^\psi$ satisfy the hypotheses of Proposition~\ref{prop:amalgmated_Ueda} with respect to $E$ and $G$. Additionally, the assumption $p N_2 p\neq \C p$ implies we can find a projection $q\in p N_2 p \setminus\{ 0, p\}$ so that
    \[
        y:=\varphi(p-q) q - \varphi(q)(p-q)
    \]
is invertible in $p M_2 p$ and satisfies
    \[
        E[y] = \frac{\varphi(y)}{\varphi(p)} p = 0.
    \]
Therefore the final statement in Proposition~\ref{prop:amalgmated_Ueda} yields
    \[
        [p(N_1 \vee N_2)p]' \cap (pMp)^\omega \subset \{y, u, v\}'\cap (pMp)^\omega \subset (pBp)^\omega = (\C p)^\omega = \C p,
    \]
as claimed.
\end{proof}

\begin{thm}[{Theorem~\ref{introthm:A}}]\label{thm:irreducibility_in_free_products}
Consider the free product
    \[
        (M,\vphi):=(M_1,\vphi_1) *(M_2,\vphi_2).
    \]
Let $N:=N_1\vee N_2$ for non-trivial subalgebras $N_i\leq M_i$ admitting $\vphi_i$-preserving faithful normal conditional expectations, 
$i=1,2$, and let $z\in N$ be the central support of the diffuse summand of $N$. Then the following are equivalent:
    \begin{enumerate}[label=(\roman*)]
        \item $\dim(N_1)+\dim(N_2)\geq 5$;
        \item $Nz$ is a factor;
        \item $(Nz)'\cap (z M z)^\omega = \C z$ for every free ultrafilter $\omega\in \beta\N\setminus \N$.
    \end{enumerate}
\end{thm}
\begin{proof}
The contrapositive of (ii)$\Rightarrow$(i) follows from \cite[Theorem 1.1]{Dyk93}, and (iii)$\Rightarrow$(ii) is immediate. So we need only show (i)$\Rightarrow$(iii). Given $\dim(N_1)+\dim(N_2)\geq 5$, we may assume without loss of generality that $\dim(N_1)\geq 3$. First, if either $N_1$ or $N_2$ is diffuse, then $z=1$ and (iii) follows from Theorem~\ref{thm:weak_spectral_amalgamated} applied to $p=1$ that $N' \cap M^\omega =\C$. So we will from now on assume that both $N_1$ and $N_2$ have a non-trivial atomic summand, which we note implies $N_1^{\vphi_1} \neq \C \neq N_2^{\vphi_2}$ by Lemma~\ref{lem:ergodic_states}.

Now, fix a non-trivial projection $p\in N_1^{\vphi_1}\setminus \{0,1\}$ and set $A:=\C p + \C (1-p)$. Observe that
    \[
        E[x]:=\frac{\vphi(xp)}{\vphi(p)} p + \frac{\vphi(x(1-p))}{\vphi(1-p)}(1-p) \qquad x\in M
    \]
gives the unique $\vphi$-preserving conditional expectation $E\colon M\to A$. Then by Lemma~\ref{lem:amalgamated_osmosis} one has
    \begin{align*}
        (M,E)&\cong (M_1, E|_{M_1}) \underset{A}{*} (\widetilde{M}_2, E|_{\widetilde{M}_2})\\
        (N, E|_{N})&\cong (N_1, E|_{N_1}) \underset{A}{*} (\widetilde{N}_2, E|_{\widetilde{N}_2}),
    \end{align*}
where $\widetilde{M}_2:=M_2\vee A$ and $\widetilde{N}_2:= N_2 \vee A$. Note that
    \[
        (\widetilde{N}_2, \vphi|_{\widetilde{N}_2}) \cong (N_2,\vphi_2) * (A,\vphi_1).
    \]
Let $q$ be the central support of the diffuse summand of $N_2^{\vphi_2}\vee A$. Since $N_2^{\vphi_2}$ and $A$ are each at least $2$-dimensional, $pq$ and $(1-p)q$ are both non-zero by \cite[Theorem 1.1]{Dyk93}. Set $e:=pq$ and, using the diffuseness of $(1-p)q (N_2^{\vphi_2} \vee A)(1-p)q$, we can choose $f\in N_2^{\vphi_2} \vee A$ to be a non-zero projection that is strictly smaller than $(1-p)q$. Noting $N_2^{\vphi_2}\vee A \subset M^{\vphi}$, it follows that $B:=\C e + \C(p-e) + \C f + \C (1-p -f)$ admits a unique $\vphi$-preserving conditional expectation $F\colon M\to B$.  
Applying Lemma~\ref{lem:amalgamated_osmosis} again, we have
    \begin{align*}
        (M,F)&\cong (\widetilde{M}_1, F|_{\widetilde{M}_1}) \underset{B}{*} (\widetilde{M}_2, F|_{\widetilde{M}_2})\\
        (N, F|_{N})&\cong (\widetilde{N}_1, F|_{\widetilde{N}_1}) \underset{B}{*} (\widetilde{N}_2, F|_{\widetilde{N}_2}),
    \end{align*}
where $\widetilde{M}_1:= M_1\vee B$ and $\widetilde{N}_1:= N_1 \vee B$. Now, observe that $e$ is minimal and central in $B$ and $e \widetilde{N}_2 e$ is diffuse since it expects onto $e (N_2^{\vphi_2}\vee A )e$ (see \cite[Lemma 1.4]{HJKEN24}). We will argue below that $e \widetilde{N}_1 e \neq \C e$, in which case Theorem~\ref{thm:weak_spectral_amalgamated} will imply $e N e \leq (e M e)^\omega $ is irreducible. We further claim that $z$ is the central support of $e$ in $N$, so that Lemma~\ref{lem:irred_corners} then completes the proof of the present theorem. Indeed, $e N e$ is diffuse since it expects onto $e\widetilde{N}_2 e$ and $N(1-z)$ is finite dimensional by \cite[Theorem 4.1]{Ued11a}, so the common corner $ e N e (1-z)$ must therefore be zero. That is, $e(1-z)=0$ and hence $e\leq z$. But then $z$ is necessarily the central support of $e$ in $N$ because $Nz$ is a factor by \cite[Theorem 4.1]{Ued11a}.

It remains to show that $e \widetilde{N}_1 e \neq \C e$. Observe that $E[e]=\frac{\vphi(e)}{\vphi(p)} p \neq 0$, and so it suffices to find a non-zero $x\in e \widetilde{N}_1 e$ with $E[x]=0$. We first consider when one of $p$ or $1-p$ is non-minimal in $N_1$. Reversing the roles of $p$ and $1-p$ if necessary, we may assume without loss of generality that $p$ is non-minimal, in which case we can find a non-zero element $a\in p N_1 p$ satisfying $\vphi(a)=0$ (and hence $E[a]=0$). Then $x:= e a e \in  e\widetilde{N}_1 e$ satisfies $E[x]=0$ by the freeness of $N_1$ and $B$ with amalgamation over $A$. It is also non-zero since, using freeness with amalgamation again, one has
    \[
        E[x^* x] = E[e a^* e a e] = \frac{\vphi(e)}{\vphi(p)} E[e a^* a e]= \frac{\vphi(e)\vphi(a^*a)}{\vphi(p)^2} E[e] = \frac{\vphi(e)^2 \vphi(a^*a)}{\vphi(p)^3} p \neq 0.
    \]
Next suppose $p$ and $1-p$ are both minimal in $N_1$. Since $\dim(N_1)\geq 3$, we must have $p N_1 (1-p)\neq 0$ lest $N_1\cong \C^2$. Thus we can find a partial isometry $v\in N_1$ with $v^*v \leq p$ and $vv^*\leq 1-p$. Recalling that we choose $f$ to be strictly smaller than $(1-p)q$ (and hence strictly smaller than $1-p$), we have that
    \[
        y:= f - E[f] = f - \frac{\vphi(f)}{\vphi(1-p)}(1-p)
    \]
is non-zero but satisfies $E[y]=0$. Then $x:=e v^* y v e \in e \widetilde{N}_1 e$ satisfies $E[x]=0$ by freeness with amalgamation since $E[v] =0$ by virtue of $v=(1-p)vp$. On the other hand, one has
    \begin{align*}
        E[x^*x]&=E[ev^* y^* v e v^* yve] = \frac{\vphi(e)}{\vphi(p)} E[e v^* y^* vv^* y ve] = \frac{\vphi(e)\vphi(vv^*)}{\vphi(p)\vphi(1-p)} E[e v^* y^*y v e]\\
            &= \frac{\vphi(e)\vphi(vv^*)\vphi(y^*y)}{\vphi(p)\vphi(1-p)^2}E[e v^* v e]=\frac{\vphi(e)\vphi(vv^*)\vphi(y^*y)\vphi(v^*v)}{\vphi(p)^2\vphi(1-p)^2}E[e] = \frac{\vphi(e)^2\vphi(vv^*)\vphi(y^*y)\vphi(v^*v)}{\vphi(p)^3\vphi(1-p)^2} p \neq 0,
    \end{align*}
so that $x$ is non-zero.
\end{proof}

\section{Centralizers of free product states}

Given a statial von Neumann algebra $(M,\vphi)$, recall that $M^{(\vphi,\text{ap})}=M$ if and only if $\vphi$ is almost periodic. On the other extreme, $M^{(\vphi,\text{ap})}=\C$ if and only if $\vphi$ is \emph{ergodic} in the sense that $M^\vphi=\C$; more generally, one has the following:

\begin{ex}\label{ex:atomic_abelian_centralizer}
For a statial von Neumann algebra $(M,\varphi)$, if one has
    \[
        (M^\vphi,\vphi) \cong \overset{p_1}{\underset{\vphi(p_1)}{\C}}\oplus \cdots \oplus \overset{p_d}{\underset{\vphi(p_d)}{\C}},
    \]
then there is a state-preserving embedding
    \[
        (M^{(\varphi,\text{ap})},\vphi) \hookrightarrow \overset{p_1,\ldots, p_d}{\underset{\vphi(p_1),\ldots, \vphi(p_d)}{M_d(\C)}}.
    \]
Indeed, given $x \in M^{(\varphi,\lambda)}$ and $1\leq i \leq d$, let $y = p_ixp_i$. Suppose $y$ is non-zero. Then $y^*y\in p_i M^{\varphi} p_i = \C p_i$ so that $y^*y = rp_i$ for some $r >0$. Set $z := \frac{1}{\sqrt{r}}y$ and observe that $zz^*z = zp_i = z$; that is, $z$ is a partial isometry. Since $z\in  p_i M^{(\vphi,\lambda)} p_i$, one has $z^*z, zz^*\in p_i M^{\vphi} p_i = \C p_i$ and therefore $z^*z = zz^* = p_i$. But then
    \[
        \vphi(p_i) = \vphi(zz^*) = \lambda \vphi(z^*z) =\lambda \varphi(p_i)
    \]
implies $\lambda =1$. Therefore $p_i M^{(\vphi, \eig)} p_i = p_i M^\vphi p_i = \C p_i$, and so each $p_i$ is a minimal projection in $M^{(\vphi,\text{ap})}$. Since $\sum_{i=1}^d p_i = 1$, the result follows.
$\hfill\blacksquare$
\end{ex}

\begin{prop}\label{prop:almost_periodic_part_free_product}
For a free product
    \[
        (M,\vphi) = (M_1,\vphi_1)* (M_2,\vphi_2),
    \]
one has
    \[
        (M^{(\vphi,\text{ap})},\vphi) \cong ( M_1^{(\vphi_1, \text{ap})},\vphi_1)* (M_2^{(\vphi_2,\text{ap})}, \vphi_2).
    \]
\end{prop}
\begin{proof}
Since the free product of almost periodic states is almost periodic, we have $M_1^{(\vphi_1, \text{ap})} \vee M_2^{(\vphi_2,\text{ap})} \leq M^{(\vphi,\text{ap})}$. For the reverse inclusion, it suffices to show that the strongly continuous representation $\R\ni t \mapsto \Delta_\vphi^{it}$ has no eigenvectors (i.e. is weakly mixing) in
    \[
        L^2(M,\vphi)\ominus L^2(M_1^{(\vphi_1, \text{ap})} \vee M_2^{(\vphi_2,\text{ap})}, \vphi).
    \]
Decomposing $L^2(M,\vphi)$ in the usual way, we see that the above orthogonal complement is given by
    \[
        \bigoplus_{d=1}^\infty \bigoplus_{i_1\neq \cdots \neq i_d} K_{i_1} \otimes \cdots \otimes K_{i_d},
    \]
where each $K_{i_j}$ is either $L^2(M_{i_j}^{(\vphi_{i_j},\text{ap})},\vphi_{i_j})$ or its orthogonal complement in $L^2(M_{i_j},\vphi_{i_j})$ with the latter occurring for at least one of $i_1,\ldots, i_d$. Since $\Delta_\vphi^{it}$ restricts to
    \[
        \Delta_{\vphi_{i_1}}^{it}\otimes \cdots \otimes \Delta_{\vphi_{i_d}}^{it}
    \]
on $K_{i_1} \otimes \cdots \otimes K_{i_d}$, and since $\Delta_{\vphi_{i_j}}^{it}$ is weakly mixing on $L^2(M_{i_j},\vphi_{i_j})\ominus L^2(M_{i_j}^{(\vphi_{i_j},\text{ap})},\vphi_{i_j})$, it follows that $\Delta_\vphi^{it}$ is weakly mixing on the claimed subspace.
\end{proof}

\begin{thm}[{Theorem~\ref{introthm:B}}]\label{thm:full_centralizer}
Consider the free product
    \[
        (M,\vphi):=(M_1,\vphi_1) *(M_2,\vphi_2),
    \]
where $\vphi_1$ and $\vphi_2$ are non-ergodic states. Let $z\in M^\vphi$ be the central support of the diffuse summand of $M^\vphi$. Then the following are equivalent:
    \begin{enumerate}[label=(\roman*)]
        \item $\dim(M_1^{(\vphi_1,\text{ap})})+\dim(M_2^{(\vphi_2,\text{ap})}) \geq 5$;
        \item $M^\vphi z$ is a factor;
        \item $(M^\vphi z)'\cap (z M z)^\omega= \C z$ for every free ultrafilter $\omega\in \beta\N\setminus \N$.
    \end{enumerate}
\end{thm}
\begin{proof}
Since $\vphi_1, \vphi_2$ are non-ergodic one has $\dim(M_i^{(\vphi_i,\text{ap})}) \geq \dim(M_i^{\vphi_i})\geq 2$ for each $i=1,2$. Thus if (i) fails then $M_i^{(\vphi_i,\text{ap})}=M_i^{\vphi_i}\cong \C^2$ for each $i=1,2$, in which case Proposition~\ref{prop:almost_periodic_part_free_product} implies
    \[
        M^\vphi \leq M^{(\vphi,\text{ap})} = M_1^{\vphi_1} * M_2^{\vphi_2} \leq M^\vphi.
    \]
Thus $M^\vphi = M_1^{\vphi_1} * M_2^{\vphi_2} = \C^2 * \C^2$, which is never a factor by \cite[Theorem 1.1]{Dyk93}, and so we have shown (ii)$\Rightarrow$(i). The implication (iii)$\Rightarrow$(ii) is immediate, and so we are left with showing (i)$\Rightarrow$(iii).

From Remark~\ref{rem:almost_periodic_diffuse_and_atomic}, we see that $z$ is also the central support of the diffuse summand of $M^{(\vphi, \text{ap})}$, and Proposition~\ref{prop:almost_periodic_part_free_product} along with \cite[Theorem 2.1]{Ued11b} imply $M^{\vphi} z$ is a factor. Thus, in order to establish (iii), it suffices by Lemma~\ref{lem:irred_corners} to show $(e M^\vphi e)' \cap (e M e)^\omega = \C e$ for any non-zero projection $e \in M^\vphi z$.

If the stronger condition $\dim(M_1^{\vphi_1})+\dim(M_2^{\vphi_2}) \geq 5$ holds, then letting $e$ be the central support of the diffuse summand of $M_1^{\vphi_1} * M_2^{\vphi_2}$ we necessarily have $e\leq z$, and Theorem~\ref{thm:irreducibility_in_free_products} gives
    \[
        (e M^\vphi e)'\cap (e M e)^\omega \subset ( (M_1^{\vphi_1} * M_2^{\vphi_2}) e)' \cap (e M e)^\omega = \C e.
    \]
So, it remains to consider the case when $\dim(M_1^{\vphi_1})=\dim(M_2^{\vphi_2}) =2$, but $\dim(M_i^{(\vphi_i,\text{ap})}) \geq 3$ for at least one of $i=1,2$. We shall assume, without loss of generality, that the almost periodic part of $(M_1,\vphi_1)$ is at least three dimensional. We argue precisely as in the proof of Theorem~\ref{thm:irreducibility_in_free_products} with $N_i:= M_i^{(\vphi_i, \text{ap})}$ and $A:= M_1^{\vphi_1}= \C p + \C (1-p)$, to obtain
    \[
        (M,F)\cong (M_1\vee B, F|_{M_1 \vee B}) \underset{B}{*} (M_2 \vee A, F|_{M_2\vee A})
    \]
for $B=\C e + \C (p-e) + \C f + \C(1-p-f)$ with $e,f\in M^\vphi$ and $F\colon M\to B$ the unique $\varphi$-preserving conditional expectation. Additionally, we have $e(N_2^{\vphi_2}\vee A)e = e( M_2^{\vphi_2} \vee M_1^{\vphi_1})e$ is diffuse. We wish to apply Theorem~\ref{thm:weak_spectral_amalgamated} to $Q_1:= (N_1 \vee B)^\vphi$ and $Q_2:= M_1^{\vphi_1} \vee M_2^{\vphi_2}$ and the projection $e\in B$, and so it remains to show $e Q_1 e \neq \C e$. Our assumptions on the dimensions of the centralizer and almost periodic part of $(M_1, \vphi_1)$ imply that there exists a partial isometry $v\in M_1^{(\vphi_1, \vphi(1-p)/\vphi(p) )}$ satisfying $v^*v=p$ and $vv^* = 1-p$ (see Example~\ref{ex:atomic_abelian_centralizer}). Setting $x:= e v^*(f - E[f]) v e$ (where $E\colon M\to A = M_1^{\vphi_1}$ is the unique $\vphi$-preserving conditional expectation), one sees that $x\not\in \C e$ by the same argument as in the proof of Theorem~\ref{thm:irreducibility_in_free_products}. Observing that $x\in e Q_1 e$, we can therefore apply Theorem~\ref{thm:weak_spectral_amalgamated} to conclude
    \[
        (e M^\vphi e)' \cap (e M e)^\omega \subset (e[ Q_1 \vee Q_2]e)' \cap (e Me)^\omega = \C e,
    \]
where the first inclusion follows from $Q_1, Q_2\leq M^\vphi$. We must also have $e\leq z$ since $ e Q_2 e \leq e M^\vphi e$ is diffuse, and so, as noted above, this concludes the proof.
\end{proof}

\begin{rem}\label{rem:diffuse_free_product_centralizer}
In the context of Theorem~\ref{thm:full_centralizer}, $z$ is also the central support of the diffuse summand of $M^{(\vphi,\text{ap})}$ by Remark~\ref{rem:almost_periodic_diffuse_and_atomic}. Thus, using Proposition~\ref{prop:almost_periodic_part_free_product} and \cite[Theorem 4.1]{Ued11a}, one can compute the finite dimensional corner $M^{(\vphi,\text{ap})}(1-z)$, along with the restriction of $\vphi$, which determine $M^\vphi(1-z)$. A necessary condition for $1-z$ to be non-zero is that each of $M_1^{(\vphi_1,\text{ap})}$ and $M_2^{(\vphi_2,\text{ap})}$ admit finite dimensional summands, which will correspond to finite dimensional summands of $M_1^{\vphi_1}$ and $M_2^{\vphi_2}$ by Remark~\ref{rem:almost_periodic_diffuse_and_atomic}. By \cite[Remark 4.2]{Ued11a}, one has $z=1$ when neither $M_1^{(\vphi_1,\text{ap})}$ nor $M_2^{(\vphi_2,\text{ap})}$ has a non-trivial, minimal, central projection with mass strictly greater than $\frac12$. Since such central projections are contained in the centralizer, it follows that a sufficient condition for $z=1$ is that neither $M_1^{\vphi_1}$ nor $M_2^{\vphi_2}$ has a non-trivial, minimal, central projection with mass strictly greater than $\frac12$.$\hfill\blacksquare$
\end{rem}

\section{Graph product von Neumann subalgebras}

In this section, we will apply Theorems~\ref{introthm:A} and \ref{introthm:B} to graph products
    \[
        (M,\vphi) = \gp_{v\in \mathcal{V}} (M_v,\vphi_v)
    \]
of families of statial von Neumann algebras over finite simple graphs $\mathcal{G}=(\mathcal{V}, \mathcal{E})$. For $v\in \mathcal{V}$ we denote $S(v):=\{w\in \mathcal{V}\colon v\sim w\}$ and $B(v):=S(v)\cup \{v\}$. 

Following the convention of \cite{CdSHJKEN25}, we always assume that $M_v^{\vphi_v}\cap \ker{\vphi_v}$ contains a unitary for each $v\in \mathcal{V}$. This assumption typically ensures $M$ has no atomic summand, which allowed for complete characterizations of factoriality and fullness of $M$ in \cite{CdSHJKEN25}, and it will likewise be helpful for our purposes. The following observation will be needed in our applications of Theorems~\ref{introthm:A} and \ref{introthm:B} below.

    

\begin{rem}\label{rem:upgrading_lemma_B1}
Recall that \cite[Lemma B.1]{CdSHJKEN25} states that for a statial von Neumann algebra $(B,\vphi)$ and an amalgamated free product
    \[
        (M,E)=(M_1, E_1) \underset{B}{*}(M_2,E_2),
    \]
one has $M'\cap M^\omega \subset B^\omega$ for any free ultrafilter $\omega\in \beta\N\setminus \N$ provided there are unitaries $u_1\in M_1^{\vphi\circ E_1}\cap \ker{E}$ and $u_2,u_3\in M_2^{\vphi\circ E_2}\cap \ker{E}$ satisfying $E[u_2^* u_3]=0$. But in fact, a careful reading of the proof of this lemma actually yields
    \[
        N'\cap M^\omega \subset \{u_1,u_2,u_3\}'\cap M^\omega \subset B^\omega,
    \]
for any $N\leq M$ containing $\{u_1,u_2,u_3\}$. In particular, this holds for $N=M^{\vphi\circ E}$.$\hfill\blacksquare$
\end{rem}

\begin{thm}\label{thm:irreducible_graph_products} 
Consider the graph product
    \[
        (M,\vphi) = \gp_{v\in \mathcal{V}} (M_v,\vphi_v)
    \]
over a finite simple graph $\mathcal{G}=(\mathcal{V}, \mathcal{E})$. Let $N:=\bigvee_{v\in \mathcal{V}} N_v$ for subalgebras $N_v\leq M_v$ admitting $\vphi_v$-preserving faithful normal conditional expectations. Suppose $N_v^{\vphi_v}\cap \ker{\vphi_v}$ contains a unitary for each $v\in \mathcal{V}$. Then $N\leq M$ is irreducible if and only if:
    \begin{enumerate}
        \item $N_v\leq M_v$ is irreducible whenever $v\in \mathcal{V}$ has $B(v)=\mathcal{V}$;
        \item $\dim(N_v)+\dim(N_w)\geq 5$ whenever $v,w\in \mathcal{V}$ have $S(v)=S(w)=\mathcal{V}\setminus\{v,w\}$.
    \end{enumerate}
\end{thm}
\begin{proof}
Let $\mathcal{G}=\mathcal{G}_1+\cdots +\mathcal{G}_d$ be a decomposition of into join-irreducible subgraphs $\mathcal{G}_j=(\mathcal{V}_j, \mathcal{E}_j)$ so that
    \[
        M \cong \underset{1\leq j\leq d}{\overline{\bigotimes}} \left(\bigvee_{v\in \mathcal{V}_j} M_v\right) \qquad \text{ and } \qquad N \cong \underset{1\leq j\leq d}{\overline{\bigotimes}} \left(\bigvee_{v\in \mathcal{V}_j} N_v\right)
    \]
(see \cite[Proposition 2.1]{CdSHJKEN25}). Consequently, $N\leq M$ is irreducible if and only if $Q_j:=\bigvee_{v\in \mathcal{V}_j} N_v \leq \bigvee_{v\in \mathcal{V}_j} M_v=: P_j$ is irreducible for each $j=1,\ldots, d$. If $|\mathcal{V}_j| \geq 3$, then $Q_j\leq P_j$ is irreducible by the proof of \cite[Theorem 2.4]{CdSHJKEN25} but using Remark~\ref{rem:upgrading_lemma_B1} instead of \cite[Lemma B.1]{CdSHJKEN25}. Thus $N\leq M$ is irreducible if and only if $Q_j \leq P_j$ is irreducible for each $1\leq j\leq d$ with $|\mathcal{V}_j|\leq 2$. One has $\mathcal{V}_j=\{v\}$ if and only if $B(v)=\mathcal{V}$, so that $Q_j= N_v$ and $P_j = M_v$, and one has $\mathcal{V}_j =\{v,w\}$ if and only if $S(v)=S(w)=\mathcal{V}\setminus\{v,w\}$, so that $Q_j = N_v * N_w$ and $P_j= M_v * M_w$. Thus $Q_j \leq P_j$ is irreducible for each $1\leq j \leq d$ with $|\mathcal{V}_j|=1$ if and only if (1) holds. Additionally, our assumption that $N_v^{\vphi_v}$ and $N_w^{\vphi_w}$ each contain a state-zero unitary implies the minimal projections in these centralizers have mass at most $\frac12$ (see \cite[Corollary A.3]{CdSHJKEN25}), and therefore $N_v*N_w$ has no atomic summand by \cite[Remark 4.2]{Ued11a}. Thus Theorem~\ref{introthm:A} implies that $Q_j \leq P_j$ is irreducible for each $1\leq j \leq d$ with $|\mathcal{V}_j|=2$ if and only if (2) holds. Hence the irreducibility of $N\leq M$ is equivalent to (1) and (2).
\end{proof}

In the previous theorem, one might hope to upgrade the irreducibility of $N\leq M$ to $N\leq M^\omega$ for any free ultrafilter $\omega\in \beta\N \setminus \N$ by replacing (1) with the assumption that $N_v \leq M_v^\omega$ is irreducible. But, using the notation of the above proof, this would only yield
    \[
        (Q_1\bar\otimes \cdots \bar\otimes Q_d)' \cap (P_1^\omega \bar\otimes \cdots \bar\otimes P_d^\omega) = (Q_1'\cap P_1^\omega)\bar\otimes \cdots \bar\otimes (Q_d' \cap P_d^\omega) = \C.
    \]
Since the $P_1^\omega \bar\otimes \cdots \bar\otimes P_d^\omega$ can be significantly smaller than $(P_1\bar\otimes \cdots \bar\otimes P_d)^\omega$ (see, for example, \cite[Proposition 3.17.(1)]{GKEPT25}), one does not immediately have $N'\cap M^\omega = \C$. This same issue will affect the conclusion of Theorem~\ref{thm:extremal_graph_products} below, which is our application of Theorem~\ref{introthm:B}. Before proceeding to this, we first record the following strengthening of \cite[Theorem 2.4]{CdSHJKEN25} for join-irreducible graphs with at least three vertices.

\begin{prop}\label{prop:join-irreducible_extremal}
Let $\mathcal{G}=(\mathcal{V},\mathcal{E})$ be a join-irreducible, finite, simple graph with $|\mathcal{V}|\geq 3$, and consider a graph product $(M,\vphi) = \gp_{v\in \mathcal{V}} (M_v,\vphi_v)$ over $\mathcal{G}$. Suppose $M_v^{\vphi_v}\cap \ker{\vphi_v}$ contains a unitary for each $v\in \mathcal{V}$. Then $(M^\vphi)'\cap M^\omega = \C$ for any free ultrafilter $\omega\in \beta\N\setminus \N$, and
    \[
        M \text{ is a factor of type }\begin{cases}
            \mathrm{II}_1 & \text{if }G=\{1\}\\
            \mathrm{III}_\lambda & \text{if }G = \lambda^\Z \\
            \mathrm{III}_1 & \text{if } G= \R_+
        \end{cases},
    \]
where $G\leq \R_+$ is the closed subgroup generated by the spectra of the modular operators of $\vphi_v$ for all $v\in \mathcal{V}$.
\end{prop}
\begin{proof}
The first claim follows from the proof of \cite[Theorem 2.4]{CdSHJKEN25}, but with Remark~\ref{rem:upgrading_lemma_B1} replacing each invocation of \cite[Lemma B.1]{CdSHJKEN25}. To determine the (sub)type of $M$, we must compute the modular spectrum $\S(M)$, and, up to removing zero, this is given by Connes spectrum $\Gamma(\sigma^\vphi)$ of the modular automorphism group of $\vphi$ (see the proof of \cite[Theorem XII.1.6]{Tak03}). The factoriality of $M^\vphi$ implies the Connes spectrum of $\sigma^\vphi$ equals its Arveson spectrum (see \cite[Lemma XI.2.2]{Tak03}), and the latter equals the closed subgroup generated by the Arveson spectrum of $\sigma^{\vphi_v}$, $v\in \mathcal{V}$, by a combination of \cite[Lemma XI.1.12]{Tak03} and the fact that $\sigma^\vphi$ restricts to $\sigma^{\vphi_v}$ on $M_v$ for each $v\in \mathcal{V}$. Finally, \cite[Proposition XI.1.24]{Tak03} implies the Arveson spectrum of $\sigma^{\vphi_v}$ is nothing but the usual spectrum of $\Delta_{\vphi_v}$ less zero, and so we have shown $\S(M)\setminus \{0\}= G$. In the case that $G$ is non-trivial, we obtain the claimed type for $M$ through \cite[Theorem XII.1.6]{Tak03}. If $G=\{1\}$, then evidently each $\vphi_v$ (and hence $\vphi$) is tracial and thus $M$ is a $\mathrm{II}_1$ factor.
\end{proof}

\begin{thm}\label{thm:extremal_graph_products}
Consider the graph product
    \[
        (M,\vphi) = \gp_{v\in \mathcal{V}} (M_v,\vphi_v)
    \]
over a finite simple graph $\mathcal{G}=(\mathcal{V}, \mathcal{E})$. Assume that $M_v^{\vphi_v}\cap \ker{\vphi_v}$ contains a unitary for each $v\in \mathcal{V}$ and that:
    \begin{enumerate}
        \item $M_v^{\vphi_v}\leq M_v$ is irreducible whenever $v\in \mathcal{V}$ has $B(v)=\mathcal{V}$;
        \item $\dim(M_v^{(\vphi_v,\text{ap})}) + \dim(M_w^{(\vphi_w,\text{ap})}) \geq 5$ whenever $v,w\in \mathcal{V}$ have $S(v) = S(w) = \mathcal{V}\setminus \{v,w\}$.
    \end{enumerate}
Then $M^\vphi\leq M$ is irreducible.
\end{thm}
\begin{proof}
Arguing as in the proof of Theorem~\ref{thm:irreducible_graph_products}, we decompose
    \[
        (M,\vphi) = \left( P_1\bar\otimes \cdots \bar\otimes P_d, \vphi_1\otimes \cdots \otimes \vphi_d \right),
    \]
where each $(P_j,\vphi_j)$ corresponds to a join-irreducible subgraph $\mathcal{G}_j=(\mathcal{V}_j, \mathcal{E}_j)$ in the decomposition of $\mathcal{G}$. We claim that $P_j^{\vphi_j}\leq P_j$ is irreducible for each $j=1,\ldots, d$. If $\mathcal{V}_j$ consists of either one or at least three vertices, then $P_j^{\vphi_j}\leq P_j$ is irreducible by either (1) or Proposition~\ref{prop:join-irreducible_extremal}, respectively. If $|\mathcal{V}_j|=2$ so that $P_j$ is a free product, then Remark~\ref{rem:diffuse_free_product_centralizer} implies $P_j^{\vphi_j}$ has no atomic summand, and therefore (2) and Theorem~\ref{introthm:B} imply $P_j^{\vphi_j}\leq P_j$ is irreducible. With the claim in hand, we therefore have
    \[
        (M^\vphi)' \cap M \subset (P_1^{\vphi_1}\bar\otimes \cdots \bar\otimes P_d^{\vphi})'\cap (P_1\bar\otimes \cdots \bar\otimes P_d) = ( (P_1^{\vphi_1})'\cap P_1) \bar\otimes \cdots \bar\otimes ( (P_d^{\vphi_d})'\cap P_d) =\C, 
    \]
where the first inclusion follows from $P_j^{\vphi_j}\leq M^\vphi$ for each $j=1,\ldots, d$.
\end{proof}

\begin{rem}\label{rem:graph_product_type}
Under the assumptions of Theorem~\ref{thm:extremal_graph_products}, $M$ is a factor and its type can be computed as follows. Using the decomposition
    \[
        (M,\vphi) = \left( P_1\bar\otimes \cdots \bar\otimes P_d, \vphi_1\otimes \cdots \otimes \vphi_d \right),
    \]
the type of $M$ is determined by the type of $P_j$ (along with the well-known rules for how types combine under tensor products). If $\mathcal{V}_j=\{v\}$, then (1) implies the modular spectrum of $P_j= M_v$ is given by the spectrum of $\Delta_{\vphi_v}$ and so is of type $\mathrm{II}_1$ or type $\mathrm{III}_\lambda$ with $\lambda\neq 0$; if $\mathcal{V}_j=\{v,w\}$, then the type of $P_j = M_v * M_w$ is determined from \cite[Theorem 4.1]{Ued11a}; and if $|\mathcal{V}_j|\geq 3$, then the type of $P_j$ is determined from   Proposition~\ref{prop:join-irreducible_extremal}. In particular, $M$ can be of type $\mathrm{II}_1$ or type $\mathrm{III}_\lambda$ with $\lambda \neq 0$, and the former occurs if and only if $\vphi_v$ is tracial for all $v\in \mathcal{V}$.$\hfill\blacksquare$
\end{rem}

\begin{rem}
Under the assumptions of Theorem~\ref{thm:extremal_graph_products}, further assume that each $\vphi_v$ is almost periodic. Then condition (2) is simply $\dim(M_v)+\dim(M_w)\geq 5$ whenever $v,w\in \mathcal{V}$ have $S(v)=S(w)=\mathcal{V}\setminus\{v,w\}$, and in the proof of Theorem~\ref{thm:extremal_graph_products} one can cite \cite[Theorem 2.1]{Ued11b} in place of Theorem~\ref{introthm:B}. Furthermore, if each $M_v$ is separable and condition (1) is strengthened to $M_v^{\vphi_v} \leq M_v^\omega$ being irreducible for any free ultrafilter $\omega\in \beta\N\setminus \N$, then $M$ will be a separable full factor by \cite[Theorem E]{CdSHJKEN25}. In this case, \cite[Lemma 4.8]{Con74} can be used to show that $\Sd(M)=\Sd(\vphi)$, and the latter is simply the (not necessarily closed) subgroup of $\R_+$ generated by $\Sd(\vphi_v)$ for all $v\in \mathcal{V}$.$\hfill\blacksquare$
\end{rem}

Recall that a faithful normal state $\vphi$ on a von Neumann algebra $M$ is said to be \emph{extremal} if $M^\vphi$ is a factor. We conclude with a pair of instances in which the extremality of graph products of almost periodic states can be fully characterized. As can be seen from the proof of Theorem~\ref{thm:extremal_graph_products}, the key obstacle is determining when tensor product states are extremal. Thus, we begin with some observations in this direction (parts of which have previously appeared in the literature; see, for example, \cite[Lemma 5.2.(ii)]{HI24}). 

\begin{prop}\label{prop:extremal_tensor_products}
Consider the tensor product
    \[
        (M,\vphi)= (M_1\bar\otimes M_2, \vphi_1\otimes \vphi_2),   
    \]
where $M_1$ and $M_2$ are factors and $\vphi_1$ and $\vphi_2$ are almost periodic states.
    \begin{enumerate}[label=(\alph*)]
        \item If $M_2$ is type $\mathrm{I}$, then $\vphi$ is extremal if and only if $\vphi_1$ is extremal and $\Sd(\vphi_2)\subset \Sd(\vphi_1)$.

        \item If $\vphi_1$ is extremal, then $\vphi$ is extremal if and only if
            \[
                \left[\bigcup_{\lambda \in \Sd(\vphi_1)\cap \Sd(\vphi_2)} M_2^{(\vphi_2,\lambda)} \right]' \cap M_2^{\vphi_2} = \C.
            \]
    \end{enumerate}
\end{prop}
\begin{proof}
\textbf{(a):} The ``if'' and ``only if'' statements follow from \cite[Lemmas 2.1 and 2.2]{HN26}, respectively.

\noindent \textbf{(b):} We claim that
    \begin{align}\label{eqn:tensor_centralizer_center}
        (M^\vphi)'\cap M^\vphi = \C \otimes \left(\left[\bigcup_{\lambda \in \Sd(\vphi_1)\cap \Sd(\vphi_2)} M_2^{(\vphi_2,\lambda)} \right]' \cap M_2^{\vphi_2}\right),
    \end{align}
from which the claimed equivalence is immediate. First observe that since $\sigma^\vphi_t = \sigma_t^{\vphi_1}\otimes \sigma_t^{\vphi_2}$ for all $t\in \R$, it follows that $M^\vphi$ is the weak operator topology closure of the span of 
    \[
       \bigcup_{\lambda\in \Sd(\vphi_1)\cap \Sd(\vphi_2)} M_1^{(\vphi_1,\lambda)}\otimes M_2^{(\vphi_2,1/\lambda)}.
    \]
This gives the `$\supset$' inclusion in (\ref{eqn:tensor_centralizer_center}). On the other hand, we can use the extremality of $\vphi_1$ to observe that
    \[
        (M^\vphi)' \cap M \subset (M_1^{\vphi_1}\bar\otimes M_2^{\vphi_2})'\cap (M_1\bar\otimes M_2) = ( (M_1^{\vphi_1})'\cap M_1) \bar\otimes ( (M_2^{\vphi_2})'\cap M_2) = \C \bar\otimes ( (M_2^{\vphi_2})'\cap M_2).
    \]
Thus any element of $(M^\vphi)'\cap M^\vphi$ is of the form $1\otimes x$ for $x\in M_2^{\vphi_2}$---fix such an element. For $\lambda\in \Sd(\vphi_1)\cap \Sd(\vphi_2)$ let $y\in M_1^{(\vphi_1, 1/\lambda)}$ and $z\in M_2^{(\vphi_2, \lambda)}$ so that $y\otimes z \in M^\vphi$, and therefore
    \[
        y \otimes [x,z] = [(1\otimes x), y\otimes z]=0.
    \]
Since $y$ is necessarily non-zero, we can apply $\vphi_1(y^*\,\cdot\,)\otimes \text{id}$ to the above to see that $[x,z]=0$. Thus $x$ commutes with $M_2^{(\vphi_2, \lambda)}$ for each $\lambda\in \Sd(\vphi_1)\cap \Sd(\vphi_2)$, and we have therefore shown the other inclusion in (\ref{eqn:tensor_centralizer_center}).
\end{proof}

Note that one always has $M'\cap M \subset M^\vphi$, and so we do not sacrifice any generality due to our assumption that the graph product $M$ is a factor in the following characterizations of the extremality of $\vphi$. Moreover, we know precisely when $M$ is a factor by \cite[Main Theorem 0.5.(2)]{CdSHJKEN25} (and our ubiquitous assumption that the centralizer of each vertex algebra contains a state-zero unitary).

\begin{thm}\label{thm:extremal_atomic_graph_products}
Let $\mathcal{G}=(\mathcal{V},\mathcal{E})$ be an incomplete, finite, simple graph, and let
    \[
        (M,\vphi)= \gp_{v\in \mathcal{V}} (M_v,\vphi_v)
    \]
be a graph product over $\mathcal{G}$ of atomic von Neumann algebras. Suppose $M$ is a factor and $M_v^{\vphi_v}\cap \ker{\vphi_v}$ contains a unitary for each $v\in \mathcal{V}$. Denote by $\mathcal{V}_0$ the set of vertices $v\in \mathcal{V}$ with $B(v)=\mathcal{V}$, and let $\Gamma\leq \R_+$ be the group generated by $\Sd(\vphi_w)$ for $w\in \mathcal{V}\setminus \mathcal{V}_0$. 
Then $\vphi$ is extremal if and only if one has $\Sd(\vphi_v) \subset \Gamma$ for all $v\in \mathcal{V}_0$.
\end{thm}
\begin{proof}
We first remark that each $\vphi_v$ is almost periodic since $M_v$ is atomic, and hence $\vphi$ is almost periodic. Denote
    \begin{align*}
        M_0:= \bigvee_{v\in \mathcal{V}_0} M_v \qquad \text{ and }\qquad  M_1:= \bigvee_{w\in \mathcal{V}\setminus \mathcal{V}_0} M_w ,
    \end{align*}
so that $(M, \vphi)= (M_0,\vphi|_{M_0})\bar\otimes (M_1,\vphi|_{M_1})$. The factoriality of $M$ implies each of $M_0$ and $M_1$ are factors, and the former is necessarily type $\mathrm{I}$ since it is simply the tensor product of the $M_v$ with $v\in \mathcal{V}_0$. Also, $\mathcal{G}$ being incomplete implies $\mathcal{V}\setminus \mathcal{V}_0$ is non-empty, and therefore $M_1$ is non-trivial. Using the notation from the proof of Theorem~\ref{thm:irreducible_graph_products}, $M_1$ is the tensor product of the $P_j$ corresponding to join-irreducible subgraphs $\mathcal{G}_j$ with at least two vertices. Moreover, for each $P_j\leq M_1$ the restriction $\vphi|_{P_j}$ is extremal: when $\mathcal{G}_j$ has two vertices this follows from \cite[Theorem 2.1]{Ued11b} and the fact that $M$ (and hence $P_j$) is a factor, and when $\mathcal{G}_j$ has at least three vertices this follows from Proposition~\ref{prop:join-irreducible_extremal}. Consequently, $\vphi|_{M_1}$ is a tensor product of extremal states and is therefore extremal by part (b) of Proposition~\ref{prop:extremal_tensor_products}. But then part (a) of this same proposition implies  $\vphi = \vphi|_{M_0}\otimes \vphi|_{M_1}$ is extremal if and only if $\Sd(\vphi|_{M_0}) \subset \Sd(\vphi|_{M_1})$. 

Now, $\sigma_t^\vphi|_{M_v} = \sigma_t^{\vphi_v}$ for all $v\in \mathcal{V}$, implies
    \[
        \Sd(\vphi|_{M_0}) = \prod_{v\in \mathcal{V}_0} \Sd(\vphi_v),
    \]
and that $\Sd(\vphi|_{M_1})$, which is necessarily a group by the extremality of $\vphi|_{M_1}$ (see \cite[Remark 2.2]{GGLN25}), equals $\Gamma$. Thus the inclusion $\Sd(\vphi|_{M_0})\subset \Sd(\vphi|_{M_1})$ is equivalent to $\Sd(\vphi_v) \subset \Gamma$ for all $v\in \mathcal{V}_0$.
\end{proof}

\begin{thm}\label{thm:extremal_almost_periodic_graph_products}
Let $\mathcal{G}=(\mathcal{V},\mathcal{E})$ be an incomplete, finite, simple graph, and let
    \[
        (M,\vphi)= \gp_{v\in \mathcal{V}} (M_v,\vphi_v)
    \]
be a graph product over $\mathcal{G}$ of von Neumann algebras equipped with almost periodic states. Suppose that $M$ is a factor and that $M_v^{\vphi_v}\cap \ker{\vphi_v}$ contains a unitary for each $v\in \mathcal{V}$. Denote by $\mathcal{V}_0$ the set of vertices $v\in \mathcal{V}$ with $B(v)=\mathcal{V}$, let $\Gamma\leq \R_+$ be the group generated by $\Sd(\vphi_w)$ for $w\in \mathcal{V}\setminus \mathcal{V}_0$, and let $M_0:= \bigvee_{v\in \mathcal{V}_0} M_v$. Then $\vphi$ is extremal if and only if
    \[
        \left[\bigcup_{\lambda\in L} M_0^{(\vphi,\lambda)}\right]'\cap M_0^\vphi =\C,
    \]
where $L= \left( \prod_{v\in \mathcal{V}_0} \Sd(\vphi_v)\right) \cap \Gamma$.
\end{thm}
\begin{proof}
Define $M_1:=\bigvee_{w\in \mathcal{V}\setminus \mathcal{V}_0} M_w$, so that arguing as in the proof of Theorem~\ref{thm:extremal_atomic_graph_products} we have: $\vphi = \vphi|_{M_0}\otimes \vphi|_{M_1}$, $\vphi|_{M_1}$ is extremal, and
    \[
        Sd(\vphi|_{M_0}) = \prod_{v\in \mathcal{V}_0} \Sd(\vphi_v) \qquad \text{ and }\qquad \Sd(\vphi|_{M_1}) = \Gamma.
    \]
Thus $L= \Sd(\vphi|_{M_0}) \cap \Sd(\vphi|_{M_1})$ and therefore the theorem follows from Proposition~\ref{prop:extremal_tensor_products}.
\end{proof}

\begin{rem}\label{rem:almost_periodic_graph_product_type}
For $M$ as in either Theorem~\ref{thm:extremal_atomic_graph_products} or Theorem~\ref{thm:extremal_almost_periodic_graph_products}, its type can be computed as follows. We have
    \[
        M = M_0 \bar\otimes M_1 = \left(\underset{v\in \mathcal{V}_0}{\overline{\bigotimes}} M_v\right) \bar\otimes M_1,
    \]
and $M_1$ is type $\mathrm{II}_1$ or type $\mathrm{III}_\lambda$ with $\lambda \neq 0$ by the procedure outlined in Remark~\ref{rem:graph_product_type}. Then the type of $M$ is determined by those of the $M_v$ for $v\in \mathcal{V}_0$. However, unlike in Theorem~\ref{thm:extremal_graph_products} where we effectively assumed $\vphi_v$ was extremal for each $v\in \mathcal{V}_0$, $M_v$ can be semifinite without $\vphi_v$ being tracial. Consequently, $M$ has more available types than in Remark~\ref{rem:graph_product_type}, though the extremality of $\vphi$ precludes if from being type $\mathrm{III}_0$. Explicitly,  $M$ can be type $\mathrm{II}_1$, $\mathrm{II}_\infty$, or type $\mathrm{III}_\lambda$ with $\lambda\neq 0$. $\hfill\blacksquare$
\end{rem}

\bibliographystyle{amsalpha}
\bibliography{references}

\end{document}